\title{Membership and Elasticity in Certain Affine Monoids}
\author{Jackson Autry\thanks{
                  San Diego State University}
        \and
        Vadim Ponomarenko\thanks{
                      San Diego State University, corresponding author
                  }
        }
\documentclass{article}
\usepackage{graphicx,amsthm,amsmath,amsfonts}
\newtheorem{theorem}{Theorem}
\newtheorem{lemma}[theorem]{Lemma}
\newtheorem{example}[theorem]{Example}
\newtheorem{corollary}[theorem]{Corollary}
\newtheorem{proposition}[theorem]{Proposition}

\renewenvironment{proof}{{\sc Proof:}}{~\hfill QED}

\begin{document}
\newpage
\maketitle
\begin{abstract}
    For affine monoids of dimension 2 with embedding dimension 2 and 3, we study the problem of determining when a vector is an element of the monoid, and the problem of determining the elasticity of a monoid element.
\end{abstract}



\newcommand{\A}{\mathcal{A}}
\newcommand{\C}{\mathbb{C}}
\newcommand{\B}{\mathcal{B}}
\newcommand{\D}{\mathcal{D}}
\newcommand{\h}{\mathcal{H}}
\newcommand{\II}{\mathcal{I}}
\renewcommand{\S}{\mathcal{S}}
\newcommand{\Z}{\mathbb{Z}}
\newcommand{\N}{\mathbb{N}}
\newcommand{\R}{\mathbb{R}}
\newcommand{\Q}{\mathbb{Q}}
\newcommand{\LL}{\mathcal{L}}
\newcommand{\RR}{\mathcal{R}}
\newcommand{\F}{\mathcal{F}}
\newcommand{\cl}{\operatorname{cl}}
\newcommand{\supp}{\operatorname{supp}}
\newcommand{\lcm}{\operatorname{lcm}}
\newcommand{\ran}{\operatorname{ran}}
\newcommand{\norm}[1]{\| #1 \|}
\newcommand{\inner}[1]{\langle #1 \rangle}
\renewcommand{\vec}[1]{{\bf #1}}
\newcommand{\vect}[2]{\left[\begin{smallmatrix}#1\\#2\end{smallmatrix}\right]}
\newcommand{\mush}[2]{\left[\begin{smallmatrix}#1 & #2\end{smallmatrix}\right]}
\newcommand{\smush}[1]{\left[\begin{smallmatrix}#1\end{smallmatrix}\right]}
\newcommand{\<}{\langle}
\renewcommand{\>}{\rangle}
\def\multichoose#1#2{\ensuremath{\left(\kern-.3em\left(\genfrac{}{}{0pt}{}{#1}{#2}\right)\kern-.3em\right)}}

\section{Introduction}
Let $\N$ denote the set of positive integers, $\N_0$ denote the set of nonnegative integers, and $\Q^\star$ denote the set of nonnegative rational numbers adjoined with $+\infty$. An \emph{affine monoid}, $S$, is a finitely generated submonoid of $\N_0^r$, with operation $+$, for some positive integer $r$. They are of substantial interest (see, e.g., \cite{MR3741074,MR3579669,Unique_Betti}). In the remainder, we restrict to the case $r=2$.  Any affine monoid is cancellative ($\mathbf{a} + \mathbf{b} = \mathbf{a} + \mathbf{c}$ implies $\mathbf{b} = \mathbf{c}$), reduced (its only unit is $0$, the identity element), and torsion free ($k\mathbf{a} = k\mathbf{b}$ for $k \in \N$ implies $\mathbf{a}= \mathbf{b}$). Let $S$ be an affine monoid minimally generated by $\A := \{\mathbf{a}_1, \ldots, \mathbf{a}_p\} \subset \N_0^r$, that is to say $S = \langle \mathbf{a}_1,\mathbf{a}_2,\ldots, \mathbf{a}_p\rangle  := \N_0\mathbf{a}_1 + \cdots + \N_0\mathbf{a}_p$ and no proper subset of $\A$ generates $S$. We say the \emph{embedding dimension} of $S$ is $p$. 
For a general introduction to monoids and their invariants, see \cite{Overview_Invariants}.

The monoid map
$$\pi_{\A} : \N_0^p \longrightarrow S; \mathbf{u} = (u_1,\ldots,u_p) \longmapsto \sum_{i = 1}^p u_i\mathbf{a}_i$$
is sometimes known as the \emph{factorization homomorphism} associated to $\A$, and if $\pi_{\A}(\mathbf{u}) =s$, $\mathbf{u}$ is called a \emph{factorization} of $s$. For every $s\in S$, the set $\mathsf{Z}(s):= \pi_{\A}^{-1}(s)$ is called the \emph{set of factorizations} of $s$. Given $s \in S$, for $\mathbf{u} = (u_1,\ldots,u_p) \in \mathsf{Z}(s)$, define the \emph{length} of the factorization $\mathbf{u}$, to be $\vert \mathbf{u} \vert = u_1 + \cdots + u_p$, and define the \emph{set of lengths} of $s$ as $\mathsf{L}(s) = \{ \vert \mathbf{u} \vert : \mathbf{u} \in \mathsf{Z}(\mathbf{a})\}$. Define the \emph{elasticity} of $s\in S$ as $\rho(s) = \frac{\max(\mathsf{L}(s))}{\min(\mathsf{L}(s))}$, and the \emph{elasticity} of $S$ to be $\rho(S) = \sup \{\rho(s): s\in S \setminus \{0\}\}$.  The elasticity is a very important monoid invariant (see, e.g., \cite{MR3602830,MR3255016,MR3035125,MR3503387}).

The monoid elasticity $\rho(S)$ for affine monoids is known (see, e.g., \cite{philipp_2010}).  In this note, our main tool will be the function $\phi:\mathbb{Z}^2 \to \Q^\star$ given by $\phi: \vect{a}{b} \mapsto \frac{a}{b}$, with $\frac{a}{0}$ conventionally taken to be $+\infty$. Our main focus will be $S\subseteq\N_0^2$, with embedding dimension $2$ and $3$.

We will compute the elasticity of individual monoid elements. We also provide membership tests for arbitrary elements of $\N_0^2$.  We will show that for a given $s\in \N_0^2$, membership in $S$ and $\rho(s)$ are largely determined by $\phi(s)$. 

\section{Preliminaries}

We begin with the observation that $\Q^\star$ is ordered, and the semigroup operation (commonly known as the mediant) preserves this order.  This property is well-known; its proof is included for completeness.

\begin{lemma}\label{mediant} Let $a,b,c,d\in\N_0$ with $\phi(\vect{a}{b})< \phi(\vect{c}{d})$.  Then $$\phi(\vect{a}{b})< \phi(\vect{a+c}{b+d})< \phi(\vect{c}{d}).$$
\end{lemma}
\begin{proof} We prove only the nontrivial case $bd\neq 0$.  Then $ad< bc$ by hypothesis.  If we add $ab$ to both sides and divide by $b(b+d)$, we conclude $\frac{a}{b}< \frac{a+c}{b+d}$ which gives the first inequality.  If we instead add $cd$ to both sides and divide by $d(b+d)$, we get the second inequality. 
\end{proof}\\

\begin{corollary}\label{mediant-cor} Let $u,v\in\N_0^2$ with $\phi(u)< \phi(v)$.  Let $s\in\langle u,v\rangle$.  Then $\phi(u)\le \phi(s)\le \phi(v)$. 
\end{corollary}
\begin{proof} Strict inequality is lost if $s=u+u$ or similar.\end{proof}\\

Let $GL(2)$ denote the set of $2\times 2$ unimodular matrices (i.e. with determinant $\pm 1$), with entries from $\mathbb{Z}$.  Let $\mush{u}{v}$ denote the $2\times 2$ matrix whose first column is $u$, and whose second column is $v$.  Let $\smush{\A}$ denote a similar matrix whose columns are the monoid generators.

\begin{corollary}\label{matrix-cor} Let $u,v\in\N_0^2$ with $\phi(u)< \phi(v)$.  Let $s\in\langle u,v\rangle$. Let $A\in GL(2)$.  Suppose that $Au, Av\in \N_0^2$.  Then $As\in \langle Au, Av\rangle$, and either $\phi(Au)\le \phi(As)\le \phi(Av)$ or $\phi(Av)\le \phi(As)\le \phi(Au)$.
\end{corollary}
\begin{proof} Since $s\in \langle u,v\rangle$, there is some vector $w$ with $\mush{u}{v}\smush{w}=\smush{s}$.  Then $A\mush{u}{v}\smush{w}=A\smush{s}$, hence $\mush{Au}{Av}\smush{w}=\smush{As}$.  Hence $As\in \langle Au, Av\rangle$.  We apply Corollary \ref{mediant-cor} in one of two ways, depending on whether $\phi(Au)\le \phi(Av)$ or $\phi(Au)\ge \phi(Av)$.
\end{proof}\\

Given some $u=\vect{a}{b}\in\N_0^2$, we say that it is $\phi$-minimal if $\gcd(a,b)=1$; otherwise we could take a smaller $u'=\vect{a/\gcd(a,b)}{b/\gcd(a,b)}$ with $\phi(u)=\phi(u')$.  Henceforth we assume that all of our monoid generators are $\phi$-minimal.

\begin{lemma}\label{phi-equal} Let $u=\vect{a}{b},v=\vect{c}{d}\in \N_0^2$.  Suppose that both are $\phi$-minimal and $\phi(u)=\phi(v)$. Then $u=v$.
\end{lemma}
\begin{proof} If $bd=0$, then $b=d=0$ and $a=c=1$; hence, $u=v$.  Otherwise $ad=bc$.  Since $\gcd(a,b)=1$, $a|c$.  Since $\gcd(c,d)=1$, $c|a$.  Since $a,c\in\N_0$, $a=c$.  Similarly, $b=d$.\end{proof}\\

Since all monoid generators are distinct, by Lemma \ref{phi-equal}, they must also have distinct $\phi$-values.  
Henceforth, we may assume, without loss of generality, that our monoid generators are given in strictly increasing $\phi$ order.  

We now recall Hermite Normal Form, an analog of row echelon form for matrices over non-fields like $\mathbb{Z}$.  For every rectangular matrix $M$ with integer entries, there is an associated square unimodular matrix $U$ such that $UM$ is (a) upper triangular; and (b) the pivot in each nonzero row is strictly to the right of the previous row; and (c) all entries of $M$ are nonnegative integers. 
For an introduction to these and other properties of HNF, see \cite{MR1181420}.

Now, for $M=\mush{u}{v}$, applying HNF we have the first column of $UM$ as $\vect{g}{0}$, where $g$ is the $\gcd$ of the entries of $u$.  Since $u$ is $\phi$-minimal, $g=1$.  Hence, we have $UM=\left[\begin{smallmatrix}1&b\\0&a\end{smallmatrix}\right]$, with $a,b\in\N_0$.  We now consider a  row-swapped HNF, defined as $U'=\left[\begin{smallmatrix}0&1\\1&0\end{smallmatrix}\right]U$, so $U'M=\left[\begin{smallmatrix}0&a\\1&b\end{smallmatrix}\right]$.  Note that $U'u, U'v\in \N_0^2$, so by Corollary \ref{matrix-cor}, if $s\in \langle u,v\rangle$ then $\phi(U'u)\le \phi(U's)\le \phi(U'v)$.  Further, note that $0=\phi(U'u)$ and $\phi(U'v)>0$.  Henceforth we will assume without loss of generality that our first generator is $\vect{0}{1}$.

We now recall Smith Normal Form, a non-field analog of the linear algebra theorem giving invertible $U,V$ with $UMV=\left[\begin{smallmatrix}I&0\\0&0\end{smallmatrix}\right]$, a block matrix. For any rectangular matrix $M$ with integer entries, there are associated square unimodular matrices $U,V$ such that $UMV=\left[\begin{smallmatrix}D&0\\0&0\end{smallmatrix}\right]$, where $D=diag(d_1, d_1d_2, \ldots, d_1d_2\cdots d_k)$.  Of particular interest to us are the $d_i$, the so-called determinantal divisors of $M$, which satisfy that $d_i$ is the gcd of all the $i\times i$ minors of $M$.  For example, $d_1(M)$ is the gcd of all the entries of $M$.

The determinantal divisors of $M$ are not disturbed upon multiplication (on either side) by any unimodular matrix.    Further, they are not disturbed by appending a column that is a $\mathbb{Z}$-linear combination of the other columns.  For an introduction to these and other properties of SNF, see \cite{MR1694173} or \cite{MR1181420}.

Given a single generator $u$, because we have assumed it is $\phi$-minimal, the determinantal divisor $d_1(\smush{u})=1$.  Consequently, for any invertible $U'$, we must have $d_1(\smush{U'u})=1$.  In particular, applying our row-swapped HNF preserves $\phi$-minimality.

We provide our first membership test for our affine monoid, of arbitrary embedding dimension.

\begin{lemma} Let $S=\langle \A\rangle$, and let $v\in \N_0^2$. Set $M=\smush{\A}$ and $M'=\mush{\A}{v}$.  If $d_2(M)\neq d_2(M')$, then $v\notin S$.
\end{lemma}
\begin{proof} If $v\in S$, then removing the last column of $M'$ (which gives $M$) will not change the determinantal divisors.\end{proof}\\

\section{Embedding Dimension 2}

In this section, we fix the case of $S=\langle u,v\rangle$, with $u=\vect{0}{1}, v=\vect{a}{b}$, and $\gcd(a,b)=1$.  Note that $d_2(\mush{u}{v})=a$.  Consider some $s=\vect{x}{y}\in\N_0^2$.  We have proved that if $s\in S$, then $0\le\phi(s)\le\frac{a}{b}$, and that $d_2(\mush{u}{v&s})=d_2(\mush{u}{v})=a$.  It turns out that these two  necessary conditions for membership are sufficient.

\begin{theorem} With notation as above, $s\in S$ if and only if both of the following hold:\begin{enumerate} \item $0\le\frac{x}{y}\le \frac{a}{b}$; and \item $a|x$.\end{enumerate}
Further, if $s\in S$, then $\rho(s)=1$.
\end{theorem}
\begin{proof} Suppose first that $s\in \langle u,v\rangle$.  By Corollary \ref{mediant-cor}, $\phi(u)\le \phi(s)\le \phi(v)$.  Note that $d_2(\left[\begin{smallmatrix}0&a\\1&b\end{smallmatrix}\right])=a$, as the $2\times 2$ minor is $-a$.  Note also that one of the $2\times 2$ minors of $\mush{\A}{s}$   has determinant $-x$, so we must have $a|x$.

Suppose now that the two conditions hold, i.e. there is some $k\in\mathbb{N}_0$ with $x=ka$.  If $k=0$, then $s=y\vect{0}{1}$.  No other factorization is possible, as even one copy of $v$ will disturb the $0$.

Otherwise, since $\frac{x}{y}\le \frac{ka}{kb}=\frac{x}{kb}$, we must have $y\ge kb$.  Hence we may write $\vect{x}{y}=k\vect{a}{b}+(y-kb)\vect{0}{1}$, which proves $s\in \langle u,v\rangle$. No other factorization is possible, by a back-substitution-type argument:  $u$ does not affect the first coordinate, so we must have $k$ copies of $v$ and hence $y-kb$ copies of $u$. \end{proof}\\

This provides an alternate proof of the well-known fact that in embedding dimension 2,  $\rho(S)=1$.

\section{Embedding Dimension 3}

We turn now to the case of embedding dimension $3$.    Henceforth, we fix the case of $S=\langle u,v,w\rangle$, with $u=\vect{0}{1}, v=\vect{a}{b}, w=\vect{c}{d}$, $\phi(u)<\phi(v)<\phi(w)$, and $\gcd(a,b)=1=\gcd(c,d)$.   Set $M=\left[\begin{smallmatrix}0&a&c\\1&b&d\end{smallmatrix}\right]$.  We will also fix $s=\vect{x}{y}\in\N_0^2$.

We first offer a simple way to compute the determinantal divisor $d_2$ below.

\begin{lemma} With notation as above, $d_2(M)=\gcd(a,c)$.
\end{lemma}
\begin{proof} Since $\gcd(a,c)$ divides each entry of the first row of each $2\times 2$ submatrix, it divides each minor.  Hence $\gcd(a,c)|d_2(M)$.  Considering the submatrices $\left[\begin{smallmatrix}0&a\\1&b\end{smallmatrix}\right]$ and $\left[\begin{smallmatrix}0&c\\1&d\end{smallmatrix}\right]$, we find that $d_2(M)$ divides each of $a,c$.  Hence $d_2(M)|\gcd(a,c)$. \end{proof}\\

Similarly to the embedding dimension 2 case, if $s\in S$, we must have $0\le \phi(s)\le \frac{c}{d}$, and $d_2(\mush{M}{s})=d_2(\smush{M})=\gcd(a,c)$.  Further, we must have $x\in \langle a,c\rangle$, since only $v,w$ have nonzero first coordinates to contribute to $x$.  Unfortunately, in general these necessary conditions are not sufficient, as the following example demonstrates.

\begin{example}\label{sad}
Consider $u=\vect{0}{1}, v=\vect{11}{10}, w=\vect{10}{3}, s=\vect{199}{119}$.  Note that $\phi(s)<2<\phi(w)$, and that $d_2(\mush{M}{s})=d_2(\smush{M})=1$. $199$ can be factored (uniquely) in $\langle 11,10\rangle$ as $199=9\cdot 11+10\cdot 10$.  However, $9v+10w=\vect{199}{120}$.  Including $u$'s will not help, so $s\notin S$.
\end{example}

If $x\in \langle a,c\rangle$, then we can impose a restriction on its representation, as follows.

\begin{proposition}\label{alpha} Let $a,c\in \mathbb{N}$ with $\gcd(a,c)=1$.  If $x\in \langle a,c\rangle$, then there are $\alpha, \beta\in\N_0$ with $x=\alpha a + \beta c$ and $0\le \alpha < c$.
\end{proposition}
\begin{proof} Since   $x\in \langle a,c\rangle$, there are some $\alpha', \beta' \in \N_0$ with $x=\alpha' a + \beta' b$.  But also $x=(\alpha' - tc)a + (\beta'+ta)c$ for all integer $t$.  Choose $t\ge 0$ maximal with $\alpha'-tc\ge 0$, set $\alpha=\alpha'-tc, \beta=\beta'+ta$, and observe that $0\le \alpha < c$.\end{proof}\\

We will frequently use the canonical factorization of $x$ in $\langle a,c\rangle$ from Proposition \ref{alpha}, which we call $\alpha(x), \beta(x)$.

Despite the setback of Example \ref{sad}, with an additional restriction, we can solve the membership problem.  Henceforth, we add the following standing hypothesis. \begin{equation}\tag{$\star$}bc-ad=1\end{equation}  Note that $(\star)$ implies that $1=\gcd(a,b)=\gcd(a,c)=\gcd(b,d)=\gcd(c,d)=1$.  Hence, condition $(\star)$ alone implies  $\phi$-minimality on $v,w$, and also $d_2(M)=1$.


\begin{theorem} With notation as above, $s\in S$ if and only if both\begin{enumerate} \item $0\le\frac{x}{y}\le \frac{c}{d}$; and \item $x\in\langle a,c\rangle$.\end{enumerate}
\end{theorem}
\begin{proof} 
If $s\in S$, both conditions are easily seen to hold.

Suppose now that the two conditions hold.  Take $\alpha,\beta$ as in Proposition \ref{alpha}. We now prove that $y\ge \alpha b + \beta d$.  Supposing otherwise, we have $y\le\alpha b + \beta d -1$.  Since $\alpha < c$, $-\alpha>-c$, and hence $(ad-bc)\alpha > -c$.  Adding $\beta cd$ to both sides, with a bit of algebra we get $\alpha ad + \beta c d > \alpha b c +\beta c d - c$, or $\frac{\alpha a + \beta c}{\alpha b + \beta d - 1}>\frac{c}{d}$.  But then $\frac{x}{y}>\frac{c}{d}$, which contradicts hypothesis.  Hence $y\ge \alpha b + \beta d$.  Then we write $s=(y-\alpha b - \beta d)\vect{0}{1}+\alpha \vect{a}{b} + \beta \vect{c}{d}$, and hence $s\in S$.\end{proof}\\

We turn now to the elasticity problem.  The different factorizations of $s$ in $S$ all come from different factorizations of $x$ in $\langle a,c\rangle$, by the following.

\begin{lemma}\label{delta} With notation as above, given $\alpha',\beta' \in \N_0$ with $x=\alpha'a + \beta' c$, there is exactly one $\delta=\delta(\alpha',\beta')\in\Z$ with $s=\delta u + \alpha' v + \beta' w$.
\end{lemma}
\begin{proof} If $s=\delta u + \alpha' v + \beta' w$, then $y=\delta + \alpha' b + \beta' d$. We solve for $\delta$ uniquely.  If $\delta\ge 0$, then $s=\delta u + \alpha' v + \beta' w$ is a factorization of $s$ in $S$. \end{proof}\\

Henceforth, we define function $\delta(\alpha,\beta)$, applying Lemma \ref{delta} to the factorization from Proposition \ref{alpha}.

We call a factorization of $s$ \emph{extreme} if it is either of minimal or maximal length.  The extreme factorizations are given in the following theorem; there are two cases based on whether $\frac{x}{y}$ is in $(0,\frac{a}{b}]$ or $[\frac{a}{b},\frac{c}{d})$.  Recall that $\lfloor z\rfloor$ denotes the greatest integer that is less than or equal to $z$.

\begin{theorem}\label{factor} With notation as above, the extreme factorizations of $s$ are $$s=(\delta - t) u + (\alpha + ct) v + (\beta - a t) w$$
for $t=0$ and for $$t=\begin{cases} \lfloor \frac{\beta}{a} \rfloor & \frac{x}{y}\le  \frac{a}{b}\\\delta & \frac{x}{y}\ge \frac{a}{b}\end{cases}.$$
These extreme factorizations have lengths $\delta+\alpha+\beta$ and $$\begin{cases}\delta+\alpha+\beta + \lfloor \frac{\beta}{a}\rfloor (c-a-1) & \frac{x}{y}\le  \frac{a}{b}\\ \delta+\alpha + \beta + \delta(c-a-1) & \frac{x}{y}\ge \frac{a}{b}\end{cases},$$ respectively.
\end{theorem}
\begin{proof} Note that, since $\gcd(a,c)=1$, all factorizations of $x$ in $\langle a,c\rangle$  are given by $x= (\alpha + ct) a + (\beta - a t) c$, for various integer $t$.  Note that $\alpha + c t\ge 0$ precisely when $t\ge 0$, by our choice of $\alpha$.

By Lemma \ref{delta}, for each choice of $t$ there is a unique $\delta_t=\delta(\alpha + ct,\beta-at)$ with $s=\delta_t u + (\alpha+ct)v + (\beta-at)w$.  Hence $y=\delta_t + (\alpha+ct)b + (\beta-at)d=\delta_t +\alpha b + \beta d + t$, so $\delta_t=y-\alpha b - \beta d - t$.  The factorization length (of $s$ in $S$) is $(\alpha + ct)+ (\beta - at) + (y - \alpha b - \beta d -t) = (\alpha + \beta + y-\alpha b-\beta d) + t(c-a-1)$.  In particular, the length varies linearly with $t$; one extreme is when $t=0$, and the other is when $t$ is maximal.

There are two upper bounds on $t$, both of which must hold.  One is that $\beta-at\ge 0$ (else the coefficient of $w$ would not be in $\N_0$), while the other is that $0\le \delta_t=y-\alpha b - \beta d - t=\delta - t$.  Now we compare the two bounds of $\frac{\beta}{a}$ and $\delta$.  We have $\frac{\beta}{a}\le \delta$ exactly when $\alpha a b + \beta c b \le \alpha a b + \beta a d + \delta a$, which holds exactly when $xb\le ya$ or $\frac{x}{y}\le \frac{a}{b}$.  In this case, we use the $\frac{\beta}{\alpha}$ bound and get the other for free; in the other case it is the reverse.

Substituting $t=0$ and $t=\lfloor \frac{\beta}{a}\rfloor$ (or $t=\delta$), we find the lengths as above.
\end{proof}\\

Note that the sign of $c-a-1$ determines which of the two extreme factorizations is minimal and which is maximal.  In particular, we have the following.

\begin{corollary} With notation as above, if $c=a+1$, then $\rho(S)=1$.
\end{corollary}
\begin{proof} By Theorem \ref{factor}, each $s\in S$ has $|\mathsf{L}(\mathbf{s})|=1$.\end{proof}\\

\begin{corollary} With notation as above, we fix $a,b,c,d,x,\alpha,\beta$ and suppose that $\beta(x)<a$.  Then, for every $y\ge \frac{bx}{a}$, $\rho(\vect{x}{y})=1$.
\end{corollary}
\begin{proof} Our hypotheses force $\frac{x}{y}\le \frac{a}{b}$ and $\lfloor \frac{\beta}{a}\rfloor=0$.  Although $\delta$ will vary based on $y$, all factorizations of $\vect{x}{y}$ have the same length. \end{proof}\\

\section{Multiples of $s\in S$}

We now fix $s\in S$, and consider factorizations of $ks=\vect{kx}{ky}\in S$ for various $k\in\N$.  For any individual $k$, we can of course compute $\rho(ks)$ using Theorem \ref{factor}, but we seek $\rho(ks)$, or estimates thereto, for all the various choices of $k$.  We offer three such results, two specific and one general.  For convenience, we recall the sign function given by $$\textrm{sign}(z)=\begin{cases}1 & z>0\\ 0 & z=0 \\ -1 & z<0\end{cases}.$$

Our special results determine $\rho(ks)$ exactly, independently of $k$, but are for periodic values of $k$ only.  There are two, based on whether or not $\frac{x}{y}\le \frac{a}{b}$.


\begin{theorem}\label{special1} With notation as above, set $\tau=\textrm{sign}(c-a-1)$. Suppose thet $ac|k$ and  $\frac{x}{y}\le \frac{a}{b}$. Then $$\rho(ks)=\left(\frac{c}{a}\frac{ya-x(b-1)}{yc-x(d-1)}\right)^\tau.$$
\end{theorem}
\begin{proof} Let $k'\in\N$ with $k=k'ac$.  We have $\alpha(kx)=0$ and $\beta=\beta(kx)=k'ax$. We calculate $\delta=\delta(0,\beta)=ky-\beta d=ak'(cy-dx)$.  One of the extreme factorization lengths will be $\delta+\beta=ak'(cy-dx)+ak'x=ak'(cy-(d-1)x)$.  The other will be $\delta+\beta+\lfloor \frac{\beta}{a}\rfloor(c-a-1)=ak'(cy-(d-1)x)+k'x(c-a-1)$.
\end{proof}\\

We now give our second special result, for the case of $k$ a multiple of $c$ and $\frac{x}{y}\ge \frac{a}{b}$.  Note that again the elasticity is independent of $k$.

\begin{theorem}\label{special2} With notation as above, set $\tau=\textrm{sign}(c-a-1)$. Suppose thet $c|k$ and  $\frac{x}{y}\ge \frac{a}{b}$. Then $$\rho(ks)=\left(c\frac{y(c-a)-x(d-b)}{yc-x(d-1)}\right)^{\tau} .$$
\end{theorem}
\begin{proof} Let $k'\in\N$ with $k=k'c$.  We have $\alpha(kx)=0$ and $\beta=\beta(kx)=k'x$. We calculate $\delta=\delta(0,\beta)=ky-\beta d=k'(cy-dx)$.  One of the extreme factorization lengths will be $\delta+\beta=k'(cy-dx)+k'x=k'(cy-(d-1)x)$.  The other will be $\delta+\beta+\delta(c-a-1)=k'(cy-(d-1)x)+k'(cy-dx)(c-a-1)$.
\end{proof}\\

The following is a general result for all $k$.  In particular, it implies that $\rho(ks)$  is largely predicted by $\phi(s)$, with this prediction becoming more accurate as $k\to\infty$.  Note also that the limiting values agree, as expected, with the  values in Theorems \ref{special1}, \ref{special2}.

\begin{theorem}\label{general} With notation as above, set $\tau=\textrm{sign}(c-a-1)$.  Then 
$$\lim_{k\to \infty} \rho(ks)=\begin{cases}\left(\frac{c}{a}\frac{ya-x(b-1)}{yc-x(d-1)}\right)^{\tau} & \frac{x}{y}\le \frac{a}{b}\\
\left(c\frac{y(c-a)-x(d-b)}{yc-x(d-1)}\right)^{\tau} & \frac{x}{y}\ge \frac{a}{b}\end{cases}.$$
\end{theorem}
\begin{proof} We set $\alpha=\alpha(kx), \beta=\beta(kx)$, with $kx=\alpha a + \beta c$, and $0\le \alpha < c$. Note that $\beta=\frac{kx-\alpha a}{c}$.  We calculate $\delta=ky-\alpha b - \beta d=ky-\alpha b - (kx-\alpha a)\frac{d}{c} = k(y-x\frac{d}{c})-\alpha(b-\frac{ad}{c})= k(y-x\frac{d}{c})-\frac{\alpha}{c}$.

Rather than taking $\rho(ks)$ as the ratio of $\max \mathsf{L}(ks)$ to $\min \mathsf{L}(ks)$, we will instead take $\rho$ as the ratio of $\frac{1}{k}\max \mathsf{L}(ks)$ to $\frac{1}{k}\min \mathsf{L}(ks)$.  One of these will be $\frac{1}{k}(\delta+\alpha+\beta)=\frac{1}{k}\left( k(y-x\frac{d}{c})-\frac{\alpha}{c} + \alpha+\frac{kx-\alpha a}{c}\right)=y-x\frac{d-1}{c}+\frac{\alpha}{k}\frac{c-a-1}{c}$.  In the limit, the last term vanishes, leaving $y-x\frac{d-1}{c}$.

We consider the case of  $\frac{x}{y}\le \frac{a}{b}$.  The other term we will have in our ratio limit will be $\frac{1}{k}\left(\delta+\alpha+\beta + \lfloor \frac{\beta}{a}\rfloor (c-a-1)\right)=y-x\frac{d-1}{c}+\frac{\alpha}{k}\frac{c-a-1}{c}+ \frac{1}{k}\lfloor \frac{\beta}{a}\rfloor (c-a-1)$
Now, $\frac{\beta}{a}=k\frac{x}{ac}-\frac{\alpha}{c}$.  In the limit we will get  $y-x\frac{d-1}{c}+\frac{x}{ac}(c-a-1)$.  We simplify to $y-x\frac{b-1}{a}$.  This gives the first formula.

Finally, we turn to the case of $\frac{x}{y}\ge \frac{a}{b}$.  The other term we will have in our ratio limit will be $\frac{1}{k}\left(\delta+\alpha+\beta + \delta(c-a-1)\right)=y-x\frac{d-1}{c}+\frac{\alpha}{k}\frac{c-a-1}{c}+ \frac{c-a-1}{k}\left(  k(y-x\frac{d}{c})-\frac{\alpha}{c}\right)$.  In the limit we will get $y-x\frac{d-1}{c}+ (c-a-1)(y-x\frac{d}{c}) = (c-a)y- (d-b) x$.  This gives the second formula. \end{proof}\\

We close by noting that the functions appearing in Theorems \ref{special1}, \ref{special2}, and \ref{general} are quite simple, being linear  fractional transformations in the variable $\frac{x}{y}=\phi(s)$.


\end{document}